\documentclass[a4paper]{article}
\usepackage[utf8]{inputenc}
\usepackage[newzealand]{babel}
\usepackage{amsmath}
\usepackage{amsthm}
\usepackage[author-year]{amsrefs}
\usepackage{amssymb}
\usepackage{tikz}
\usepackage{verbatim}
\usepackage{graphicx}
\usepackage{caption}
\usepackage{subcaption}
\usepackage{booktabs}
\usepackage{aliascnt}
\usepackage[breaklinks=true,pdfborder={0 0 0},pdfdisplaydoctitle=true,
pdfpagelabels=true]{hyperref}
\hypersetup{pdfauthor={Stefan Kranich},pdftitle={An epsilon-delta bound
for plane algebraic curves and its use for certified homotopy continuation
of systems of plane algebraic curves}}

\theoremstyle{plain}
\newtheorem{theorem}{Theorem}[section]

\newcommand{\mynewtheorem}[2]{
    \newaliascnt{#1}{theorem}
    \newtheorem{#1}[#1]{#2}
    \aliascntresetthe{#1}
    \expandafter\def\csname #1autorefname\endcsname{#2}
}

\mynewtheorem{lemma}{Lemma}
\theoremstyle{definition}
\mynewtheorem{algorithm}{Algorithm}
\mynewtheorem{definition}{Definition}
\mynewtheorem{example}{Example}
\mynewtheorem{problem}{Problem}
\theoremstyle{remark}
\mynewtheorem{remark}{Remark}
\makeatletter
\let\c@figure\@undefined
\let\c@table\@undefined
\makeatother
\newaliascnt{figure}{theorem}
\aliascntresetthe{figure}
\numberwithin{figure}{section}
\newaliascnt{table}{theorem}
\aliascntresetthe{table}
\numberwithin{table}{section}

\title{An epsilon-delta bound for plane algebraic curves and its use
for certified homotopy continuation of systems of plane algebraic curves}
\author{Stefan Kranich\footnote{Zentrum Mathematik (M10), Technische
Universität München, 85747~Garching, Germany; E-mail address: \url{kranich@ma.tum.de}}}
\begin{document}
\maketitle

\begin{abstract}
We explain how, given a plane algebraic curve $\mathcal{C}\colon f(x,y) =
0$, $x_1 \in \mathbb{C}$ not a singularity of $y$ w.r.t.\ $x$, and
$\varepsilon > 0$, we can compute $\delta > 0$ such that $|y_j(x_1) -
y_j(x_2)| < \varepsilon$ for all holomorphic functions $y_j(x)$ which
satisfy $f(x, y_j(x)) = 0$ in a neighbourhood of $x_1$ and for all $x_2$
with $|x_1 - x_2| < \delta$.
Consequently, we obtain an algorithm for reliable homotopy continuation
of plane algebraic curves. As an example application, we study continuous
deformation of closed discrete Darboux transforms.

Moreover, we discuss a scheme for reliable homotopy continuation of
triangular polynomial systems. A general implementation has remained
elusive so far. However, the epsilon-delta bound enables us to handle
the special case of systems of plane algebraic curves. The bound helps us
to determine a feasible step size and paths, which are equivalent w.r.t.\
analytic continuation to the actual paths of the variables but along
which we can proceed more easily.
\end{abstract}

\maketitle

\section{Motivation}

\label{sec:motivation}

In many geometric problems, variables depend analytically on some
parameter. If we want to analyze and experiment with these problems
using interactive software, whenever the user continuously modifies
the parameter, we must update the dependent variables accordingly. For
many applications, in doing so, the analytical relationship between
variables and parameter should be preserved at all times. Therefore we
need reliable algorithms for analytic continuation.

Consider for example the following problem of discrete differential
geometry~\cite{Hoffmann2009}*{Section~2.6}. Let there be a regular
discrete curve $\gamma$ in $\mathbb{CP}^1$,
i.e.\ a polygonal chain with distinct vertices $\gamma_0, \gamma_1,
\dots, \gamma_n \in \mathbb{CP}^1$. We define the \emph{discrete Darboux
transform} $\tilde\gamma$ of $\gamma$ with initial point $\tilde\gamma_0
\in \mathbb{CP}^1$ and parameter $\mu \in \mathbb{C}$ as follows:
for all $j = 1, 2, \dots, n$, let $\tilde\gamma_j \in \mathbb{CP}^1$
be the unique point for which the cross-ratio \[(\gamma_{j - 1},
\gamma_j; \tilde\gamma_j, \tilde\gamma_{j-1}) := \frac{(\gamma_{j-1}
- \tilde\gamma_j)(\gamma_j - \tilde\gamma_{j-1})}{(\gamma_{j-1} -
\tilde\gamma_{j-1})(\gamma_j - \tilde\gamma_j)} = \mu.\]
It can be shown that $\tilde\gamma_{j-1}$ is mapped to $\tilde\gamma_j$
by a unique Möbius transformation, which depends only on
$\gamma_{j-1}$, $\gamma_j$, and $\mu$, but not on $\tilde\gamma_j$. Hence,
there exists a unique Möbius transformation $M$ depending on $\gamma_0$,
$\gamma_1$, $\dots$, $\gamma_n$, and $\mu$, which maps an initial point
$\tilde\gamma_0$ to the corresponding last point $\tilde\gamma_n$ of
$\tilde\gamma$. Consequently, for every choice of $\mu \in \mathbb{C}$,
there are two choices of initial point $\tilde\gamma_0$ (counted with
multiplicity) such that $\tilde\gamma$ is a closed polygonal chain. These
are exactly the fixed points of $M$ or, in other words, the roots of the
characteristic polynomial of $M$. The vanishing of the characteristic
polynomial establishes an algebraic (particularly analytical) relationship
between $\mu$ and $\tilde\gamma_0$.

If we want to study closed Darboux transforms of a discrete curve
$\gamma$ for varying parameter $\mu$ using interactive software, then
we must analytically continue $\tilde\gamma_0$. Otherwise we may observe
sudden jumps of $\tilde\gamma_0$ under continuous movement of $\mu$,
which have no mathematical justification.

In practice, of course, we cannot modify a parameter
continuously. Instead, we obtain a series of parameter values at a series
of discrete points in time. We do not know how the parameter moves between
sample points. A natural approach would be to interpolate linearly between
consecutive parameter values (using a time parameter in the unit
interval). However, the segment between parameter values may contain
singularities beyond which analytic continuation becomes impossible. Thus
it seems reasonable to analytically continue along the polygonal chain
of parameter values as long as this is possible, and to deviate from
that path otherwise. Such a deviation can still be interpreted as a
linear interpolation between consecutive parameter values if we let
the time parameter run from $0$ to $1$ on an arbitrary path through the
complex plane instead of restricting it to the unit interval.

This is the paradigm of `complex detours' invented by Kortenkamp
and Richter-Gebert for their interactive geometry software
Cinderella~\cite{KortenkampRichterGebert2006}. It is described
in more detail in~\citelist{\cite{Kortenkamp1999}*{esp.\ Chapter~7}
\cite{KortenkampRichterGebert2001b}
\cite{KortenkampRichterGebert2002}}. Essentially the same concept
was conceived in the context of homotopy continuation by Morgan and
Sommese~\cite{MorganSommese1987}, who later named it the
`gamma trick'~\cite{SommeseWampler2005}*{Lemma~7.1.3 on p.~94}.

Once we have chosen a path for the parameter, we must determine the
right value of the dependent variable at consecutive sample points.
How this can be achieved may in fact be relatively easy to see
for us---just determine values in a way such that there are no
jumps---but hard to see for an algorithm. The tracing problem of
dynamic geometry, i.e.\ tracing the positions of dependent elements
of a geometric construction under movement of a free element, is
NP-complete already for constructions that only involve points,
lines through two points, intersection of lines, and angle
bisectors~\cite{KortenkampRichterGebert2002}.

The interactive geometry software Cinderella currently uses a heuristic
for path following.
Most homotopy continuation methods use a predictor-corrector
approach, which is generally also heuristic. For an overview
of homotopy continuation methods, consider the books by
Allgower and Georg~\ycite{AllgowerGeorg1990} or Sommese and
Wampler~\ycite{SommeseWampler2005}.
Lately, certified homotopy continuation methods have
emerged~\cites{BeltranLeykin2012,BeltranLeykin2013,HauensteinSottile2012,
HauensteinHaywoodLiddell2014}.
They are based on Smale's alpha theory~\cite{Smale1986}.

In what follows, we derive a certified algorithm for analytic
continuation of plane algebraic curves based on the following simple
observation: Due to continuity, if the parameter moves little, so does
the dependent variable. Hence, if we take small enough steps along the
parameter path, we can choose the right value of the dependent variable
based on proximity.
As an application, we return to the example of continuous deformation
of closed discrete Darboux transforms.
Moreover, we show how the algorithm generalizes to systems of plane
algebraic curves.
A comparison with other approaches demonstrates the practicability of
our algorithms.

\section{Computing an epsilon-delta bound for plane algebraic curves}

\begin{theorem}
\label{thm:epsilon-delta-bound}
Let $\mathcal{C}\colon f(x, y) = 0$ be a complex plane algebraic curve,
where \[f(x, y) = \sum\limits_{k = 0}^n a_k(x) y^{n - k}\] is a polynomial
of degree $n$ in $y$ whose coefficients $a_k(x)$ are polynomials in $x$.
Let $x_1 \in \mathbb{C}$ be a point
in the complex plane at which neither the leading coefficient $a_0(x)$
nor the discriminant of $f(x,y)$ w.r.t.\ $y$ vanish.
Then for every $\varepsilon > 0$, we can algorithmically compute $\delta >
0$ such that \[|y_j(x_1) - y_j(x_2)| < \varepsilon\] for all holomorphic
functions $y_j(x)$, $j = 1, 2, \dots, n$, that satisfy $f(x, y_j(x)) = 0$
in a neighbourhood of $x_1$ and for all $x_2$ with $|x_1 - x_2| < \delta$.
\end{theorem}

\begin{remark}
\label{rem:analytic-continuation}
How does \autoref{thm:epsilon-delta-bound} help us to perform analytic
continuation? Let $\varepsilon$ be half the minimal distance between
the $y$-values at $x_1$. Then for any $x_2$ less than $\delta$ away from
$x_1$ the following holds: The $y$-value $y_j(x_2)$, which results from
analytic continuation of $y_j(x)$ along the segment from $x_1$ to $x_2$,
is closer to $y_j(x_1)$ than to any other $y$-value at $x_1$.
In other words, $\delta$ provides an upper bound for the step width
of parameter $x$ such that we may match $y$-values on the same branch
based on proximity.
\end{remark}

\noindent Our plan for the proof of \autoref{thm:epsilon-delta-bound}
is as follows: We will see that there is an upper bound of $\delta$
depending on \begin{enumerate}
\item the radius of convergence of the Taylor expansion of $y_j(x)$
at $x_1$,
\item the modulus of the derivative of $y_j(x)$ at $x_1$,
\item the maximum modulus of $y_j(x)$ on a circle centred at $x_1$,
\end{enumerate}
for $j = 1, 2, \dots, n$, respectively.
We derive a formula for that upper bound and then compute bounds for
its ingredients. To this end, we need the following lemmas.

\begin{lemma}
\label{lem:taylor}
Let $U \subset \mathbb{C}$ be an open subset of the complex plane, and
let \[y_j\colon U \to \mathbb{C}\] be holomorphic. Taylor expansion of $y_j$
around $x_1 \in U$ yields \[y_j(x_2) = y_j(x_1) + (x_2 - x_1) y_j'(x_1) +
{(x_2 - x_1)}^2 R(x_2),\] for all $x_2 \in \mathbb{C}$ such that $|x_2 -
x_1| < \rho$ and sufficiently small $\rho > 0$. The remainder $R(x_2)$
satisfies \[|R(x_2)| \leq \frac{M}{\rho (\rho - |x_2 - x_1|)}\] where \[M
= \max_{t \in \left[0,2 \pi\right]} |y_j(x_1 + \rho \mathrm{e}^{\mathrm{i} t})|.\]
\end{lemma}

\noindent \autoref{lem:taylor} is a standard result of complex
analysis~\cite{Ahlfors1979}*{p.~124--126}, which we therefore do not
prove here.

\begin{lemma}[implicit differentiation]
\label{lem:implicit-differentiation}
Let $f(x, y)$ be a complex polynomial. Let $U \subset \mathbb{C}$ be an
open subset of the complex plane. Let $y_j\colon U \to \mathbb{C}$ be a
holomorphic function that satisfies $f(x, y_j(x)) = 0$ for all $x \in
U$. Then for all $x_1 \in U$ with $f_y(x_1, y_j(x_1)) \neq 0$ it follows
that \[y_j'(x_1) = -\frac{f_x(x_1, y_j(x_1))}{f_y(x_1, y_j(x_1))}.\]
\end{lemma}

\begin{proof}
By the chain rule, the total differential of $f(x,y_j(x)) = 0$ w.r.t.\ $x$ is
\[Df(x,y_j(x)) = f_x(x,y_j(x)) + f_y(x,y_j(x)) \cdot y_j'(x) = 0.\]
Therefore
\[y_j'(x_1) = -\frac{f_x(x_1, y_j(x_1))}{f_y(x_1, y_j(x_1))}.\qedhere\]
\end{proof}

\begin{lemma}[\ocite{Fujiwara1916}*{Inequality~3 on p.~168}]
\label{lem:fujiwara}
Consider a polynomial \[p(x) = \sum\limits_{k = 0}^{n} a_k x^{n - k}\]
of degree $n$ with complex coefficients $a_k \in \mathbb{C}$, $k = 0,
1, \dots, n$. Then all $\bar{x} \in \mathbb{C}$ with $p(\bar{x}) = 0$ satisfy
\[|\bar{x}| < 2 \max \left\{\left|\frac{a_k}{a_0}\right|^\frac{1}{k}
\colon k = 1, \dots, n \right\}.\]
\end{lemma}

\begin{proof}
Consider the inequality \begin{equation}
\label{eq:fujiwara-inequality}
|p(x)| \geq |a_0| {|x|}^n - \sum_{k=1}^n |a_k| {|x|}^{n-k}.
\end{equation} The RHS of~\eqref{eq:fujiwara-inequality} is positive
if \[|a_0| {|x|}^n \geq 2^k |a_k| {|x|}^{n-k}, \quad k = 1,
2, \dots, n,\] because then \[|a_0| |x|^n > (1 - 2^{-n}) |a_0| |x|^n =
\sum_{k = 1}^n 2^{-k} |a_0| |x|^n \geq \sum_{k=1}^n |a_k| {|x|}^{n-k}.\]
Hence, $|p(x)| > 0$ if \[|x| \geq \max {\left\{2^k
\left|\frac{a_k}{a_0}\right|\right\}}^\frac{1}{k}\] and thus \[|\bar{x}|
< 2 \max \left\{\left|\frac{a_k}{a_0}\right|^\frac{1}{k} \colon k =
1, \dots, n \right\}\] for all zeros $\bar{x} \in \mathbb{C}$ of $p(x)$.
\end{proof}

\begin{lemma}[bounds for trigonometric polynomials]
\label{lem:trigonometric}
Consider a trigonometric polynomial of degree $n$ of the form
\[p(x_1 + \rho \mathrm{e}^{\mathrm{i} t}) = \sum\limits_{k = 0}^n a_k {(x_1 +
\rho \mathrm{e}^{\mathrm{i} t})}^{n - k}.\]
Then \[|p(x_1 + \rho \mathrm{e}^{\mathrm{i} t})| \leq \sum\limits_{k = 0}^n |a_k|
{(|x_1| + |\rho|)}^{n - k}.\]
Moreover, if the zeros $\bar{x}_1$, $\bar{x}_2$, $\dots$, $\bar{x}_n$
of $p(x)$ satisfy $|\bar{x}_k - x_1| > \rho$ then
\[|p(x_1 + \rho \mathrm{e}^{\mathrm{i} t})| \geq |a_0| \prod\limits_{k = 0}^{n}
(|\bar{x}_k - x_1| - \rho) > 0.\]
\end{lemma}

\begin{proof}
The upper bound follows from the triangle inequality. The lower bound
follows from the factorization \[p(x_1 + \rho \mathrm{e}^{\mathrm{i} t}) =
a_0 \prod\limits_{k = 0}^n (x_1 + \rho \mathrm{e}^{\mathrm{i} t} - \bar{x}_k)\]
and the fact that $|x_1 + \rho \mathrm{e}^{\mathrm{i} t} - \bar{x}_k| \geq |\rho
- |\bar{x}_k - x_1||$. Note that the lower bound is positive by the
assumptions that $|\bar{x}_k - x_1| > \rho$ and that $p$ has degree $n$,
i.e.\ $a_0 \neq 0$.
\end{proof}

\begin{proof}[Proof of \autoref{thm:epsilon-delta-bound}]
Let $y_j(x)$, $j = 1, \dots, n$, denote the holomorphic functions that
satisfy $f(x, y_j(x)) = 0$ in a neighbourhood of $x_1$.
By \autoref{lem:taylor},
\begin{equation}
\label{eq:taylor-expansion}
y_j(x_2) = y_j(x_1) + (x_2 - x_1) y_j'(x_1) + {(x_2 - x_1)}^2 R_j(x_2)
\end{equation}
for all $x_2 \in \mathbb{C}$ such that $|x_2 - x_1| < \rho$ and
sufficiently small $\rho > 0$.
If we bring $y_j(x_1)$ to the LHS of~\eqref{eq:taylor-expansion}, take
the absolute value on both sides, and apply the triangle inequality,
we see that
\begin{align}
\left|y_j(x_1) - y_j(x_2)\right| &= |x_2 - x_1||y_j'(x_1) + (x_2 -
x_1) R_j(x_2)|\notag\\
&\leq |x_2 - x_1|(|y_j'(x_1)| + |x_2 - x_1| |R_j(x_2)|)\notag\\
&= |R_j(x_2)| |x_2 - x_1|^2 + |y_j'(x_1)| |x_2 - x_1|.
\end{align}
Hence, under the above assumptions,
\begin{equation}
\label{eq:sufficient-bound-quadratic-form}
|R_j(x_2)| |x_2 - x_1|^2 + |y_j'(x_1)| |x_2 - x_1| - \varepsilon < 0.
\end{equation}
is a sufficient condition for $|y_j(x_1) - y_j(x_2)| < \varepsilon$.

The LHS of~\eqref{eq:sufficient-bound-quadratic-form} is strictly
increasing in $|y_j'(x_1)|$ and $|R_j(x_2)|$. Therefore, if we plug in
the bounds
\begin{equation}
\label{eq:capital-y}
|y_j'(x_1)| \leq \max_j |y_j'(x_1)| =: Y
\end{equation} and
\[|R_j(x_2)| \leq \frac{M}{\rho (\rho - |x_2 - x_1|)}\]
(see~\autoref{lem:taylor})
into~\eqref{eq:sufficient-bound-quadratic-form}, we obtain a stronger
sufficient condition for \[|y_j(x_1) - y_j(x_2)| < \varepsilon,\]
namely
\begin{align}
&\frac{M}{\rho (\rho - |x_2 - x_1|)} |x_2 - x_1|^2 + Y |x_2 - x_1| -
\varepsilon < 0 \notag\\
\Leftrightarrow\ &M |x_2 - x_1|^2 + \rho (\rho - |x_2 - x_1|) (Y |x_2
- x_1| - \varepsilon) < 0 \notag\\
\Leftrightarrow\ &(M - \rho Y) |x_2 - x_1|^2 + \rho (\rho Y
+ \varepsilon) |x_2 - x_1| - \varepsilon \rho^2 < 0.
\label{eq:sufficient-bound-plugging-in}
\end{align}

\noindent How we can transform~\eqref{eq:sufficient-bound-plugging-in}
into a sufficient bound on $|x_2 - x_1|$ depends on the sign of $M -
\rho Y$.

First case: $M - \rho Y > 0$. The LHS
of~\eqref{eq:sufficient-bound-plugging-in} describes a smile parabola
in $|x_2 - x_1|$ with a positive and a negative root. Since $|x_2 - x_1|
\geq 0$, we need only bound $|x_2 - x_1|$ from above by the positive
root, i.e.\
\begin{align*}
|x_2 - x_1| &< \frac{-\rho (\rho Y + \varepsilon) +
\sqrt{\rho^2 {(\rho Y + \varepsilon)}^2 + 4 (M - \rho
Y) \varepsilon \rho^2}}{2 (M - \rho Y)}\\
&= \frac{\rho \left(\sqrt{{(\rho Y - \varepsilon)}^2 + 4 \varepsilon
M} - (\rho Y + \varepsilon)\right)}{2 (M - \rho Y)}.
\end{align*}

\noindent Second case: $M - \rho Y < 0$. The LHS
of~\eqref{eq:sufficient-bound-plugging-in} describes a frown parabola in
$|x_2 - x_1|$ with one root greater than $\rho$ and one root between $0$
and $\rho$. Since $|x_2 - x_1| < \rho$ by definition, we need only bound
$|x_2 - x_1|$ from above by the smaller root, i.e.\
\begin{align*}
|x_2 - x_1| &< \frac{\rho (\rho Y + \varepsilon) - \rho
\sqrt{{(\rho Y + \varepsilon)}^2 - 4 (\rho Y - M)
\varepsilon}}{2 (\rho Y - M)}\\
&= \frac{\rho \left(\sqrt{{(\rho Y - \varepsilon)}^2 +
4 \varepsilon M} - (\rho Y + \varepsilon)\right)}{2 (M -
\rho Y)}.
\end{align*}

\noindent Third case: $M - \rho Y = 0$. The LHS
of~\eqref{eq:sufficient-bound-plugging-in} reduces to
\[
\rho (\rho Y + \varepsilon) |x_2 - x_1| - \varepsilon \rho^2
< 0 \quad\Leftrightarrow\quad |x_2 - x_1| < \frac{\varepsilon \rho}{\rho Y +
\varepsilon}.
\]
This bound is asymptotically equivalent to the previous bounds as $M$
approaches $\rho Y$. Altogether, we thus arrive at the
sufficient bound
\begin{equation}
\label{eq:sufficient-bound}
|x_2 - x_1| < \frac{\rho \left(\sqrt{{(\rho Y - \varepsilon)}^2 + 4
\varepsilon M} - (\rho Y + \varepsilon)\right)}{2 (M - \rho Y)}.
\end{equation}

\noindent The RHS of~\eqref{eq:sufficient-bound} has the expected
qualitative behaviour: It is strictly increasing in $\varepsilon$ and
$\rho$, and strictly decreasing in $M$ and $Y$.

It remains to be shown that we can compute bounds for the ingredients
$\rho$, $Y$, and $M$ of~\eqref{eq:sufficient-bound}.

\autoref{lem:taylor} (and thus our argument) is valid if and only if
$\rho$ is smaller than the radius of convergence of the Taylor expansion
of $y_j(x)$. Therefore, we must choose $\rho$ smaller than the distance
between $x_1$ and the singularities of $y_j(x)$, $j = 1, 2, \dots, n$.
Recall that $y_j(x)$ satisfies $f(x, y_j(x)) = 0$ in a neighbourhood of
$x_1$, where \[f(x, y) = \sum\limits_{k = 0}^n a_k(x) y^{n - k}.\]
In particular, $\rho$ must be smaller than the distance between $x_1$
and the zeros of $a_0(x)$. The zeros of $a_0(x)$ are exactly the poles
of $y_j(x)$.
The remaining finite singularities of $y_j(x)$ are exactly the finite
ramification points of $y_j(x)$. These are zeros of the discriminant of
$f(x,y)$ w.r.t.\ $y$.
Hence, we may choose any
\[\rho < \min\{|x_1 - x| \colon a_0(x) \cdot \Delta_y(f(x,y))(x) = 0\},\]
where $\Delta_y(f(x,y))(x)$ denotes the discriminant of $f$ w.r.t.\ $y$.

We can compute
\[Y = \max_j |y_j'(x_1)| = \max_j \left|\frac{f_x(x_1, y_j(x_1))}{f_y(x_1,
y_j(x_1))}\right|\]
by \autoref{lem:implicit-differentiation}. Note that the denominator does
not vanish by the assumption that $x_1$ is not a zero of the discriminant
of $f(x, y)$ w.r.t.\ $y$.

Therefore, $M$ remains to be computed or bounded from above. To that
end, we can apply \autoref{lem:fujiwara} to
\[f(x, y_j(x)) = \sum\limits_{k = 0}^n a_k(x) {y_j(x)}^k,\]
interpreted as a polynomial in $y_j(x)$.
By our choice of $\rho$, the leading coefficient $a_0(x)$ does not vanish
for all $x$ with $|x - x_1| \leq \rho$.
For those $x$ and for all $j = 1, 2, \dots, n$, \autoref{lem:fujiwara} yields
\[|y_j(x)| < 2 \max \left\{{\left|\frac{a_k(x)}{a_0(x)}\right|}^\frac{1}{k}
\mid k = 1, \dots, n\right\}.\]
Consequently, 
\[M < 2 \max_{t \in \left[0, 2\pi\right]} \left\{{\left|\frac{a_k(x_1 +
\rho \mathrm{e}^{\mathrm{i} t})}{a_0(x_1 + \rho \mathrm{e}^{\mathrm{i} t})}\right|}^\frac{1}{k} \mid k =
1, \dots, n\right\}.\]
By \autoref{lem:trigonometric}, we can compute upper bounds $\tilde{a}_k$
of $\max_{t \in \left[0,2\pi\right]} |a_k(x_1 + \rho \mathrm{e}^{\mathrm{i} t})|$ and a
lower bound $\tilde{a}_0 > 0$ of $\min_{t \in \left[0,2\pi\right]}
|a_0(x_1 + \rho \mathrm{e}^{\mathrm{i} t})|$, which are much easier to compute than these
extreme values.

The zeros of $a_0(x)$ and of $\Delta_y(f(x,y))(x)$ can be computed (at
least to arbitrary precision) using a root-finding algorithm. Similarly,
the values $y_j(x_1)$, $j = 1, 2, \dots, n$, can be computed (at least to
arbitrary precision) by solving \[f(x_1, y_j(x_1)) = 0\] for $y_j(x_1)$.

Let us summarize our argument: We may choose
\begin{equation}
\label{eq:epsilon-delta-bound}
\delta = \frac{\rho \left(\sqrt{{(\rho Y - \varepsilon)}^2 + 4 \varepsilon
M} - (\rho Y + \varepsilon)\right)}{2 (M - \rho Y)},
\end{equation}
where
\begin{gather*}
\rho < \min\{|x_1 - x| \colon a_0(x) \cdot \Delta_y(f(x,y))(x) = 0\},\\
Y:= \max\limits_j \left|\dfrac{f_x(x_1, y_j(x_1))}{f_y(x_1,
y_j(x_1))}\right|,\quad
M := 2 \max\limits_k
{\left(\frac{\tilde{a}_k}{\tilde{a}_0}\right)}^{\frac{1}{k}}. \qedhere
\end{gather*}
\end{proof}

\begin{remark}
For \autoref{thm:epsilon-delta-bound} to hold, $f(x,y)$ needs neither
be irreducible nor square-free. However, if $f(x,y)$ is not square-free,
the discriminant may vanish identically and the epsilon-delta bound is
no longer useful. If $f(x,y)$ is square-free but not irreducible, the
epsilon-delta bound for $y$-values on one irreducible component
may be smaller than necessary due to the influence of zeros of the
discriminant of other irreducible components.
\end{remark}

\section{Certified homotopy continuation of plane algebraic curves}

\noindent \autoref{thm:epsilon-delta-bound} enables us to solve the
following problem:

\begin{problem}
\label{prob:analytic-continuation}
Consider a plane algebraic curve \[\mathcal{C}\colon f(x,y) = 0.\] Let
$x\colon \left[0,1\right] \to \mathbb{C},\ t \mapsto x(t)$ be a monotonic
(distance non-decreasing) path, i.e.\
\[|x(0) - x(t_1)| \leq |x(0) - x(t_2)| \quad \text{for} \quad 0 \leq
t_1 \leq t_2 \leq 1.\]
Let $y(0) \in \mathbb{C}$ satisfy $f(x(0), y(0)) = 0$.
If analytic continuation of $y$ along $x(t)$ is possible, determine
the value $y(1)$ that results from initial value $y(0)$ under analytic
continuation of $y$ along $x(t)$.
\end{problem}

\noindent The algorithm for~\autoref{prob:analytic-continuation} follows
from \autoref{rem:analytic-continuation}:

\begin{algorithm}
\label{alg:analytic-continuation}
Let $f(x,y)$, $x(t)$, and $y(0)$ be defined as
in~\autoref{prob:analytic-continuation}.
\begin{enumerate}
\item Let $T = 0$.
\item While $T < 1$,
\begin{enumerate}
\item Let $\varepsilon$ be half the minimum distance between the $y$
with \[f(x(T), y) = 0.\]
\item Compute $\delta$ by the epsilon-delta bound
of~\autoref{thm:epsilon-delta-bound}.
\item Use bisection to maximize $T^\ast \in \left[T,1\right]$ such that
$|x(T) - x(T^\ast)| < \delta$.
\item Let $y(T^\ast)$ be the $y$ with $f(x(T^\ast), y) = 0$ closest to $y(T)$.
\item Let $T = T^\ast$.
\end{enumerate}
\item Output $y(1)$ and stop.
\end{enumerate}
\end{algorithm}

\section[Continuous deformation of closed discrete Darboux
transforms]{Case study: continuous deformation of closed discrete
Darboux transforms}
\label{sec:case-study}

\autoref{alg:analytic-continuation} shows how the epsilon-delta
bound can be used for certified homotopy continuation of plane algebraic
curves. In this section, as an example application, let us return to the
closed discrete Darboux transform introduced in \autoref{sec:motivation}.

We generally follow the exposition of~\ocite{Hoffmann2009}*{Section~2.6}
but use a slightly different definition of \emph{cross-ratio}.
(A value $\mu$ of our cross-ratio corresponds to a value $1 - \mu$
of the cross-ratio in~\cite{Hoffmann2009}*{Section~2.6} and vice versa.)

Recall the definition of discrete Darboux transform: 

\begin{definition}[discrete Darboux transform]
Let $\gamma$ be a regular discrete curve in $\mathbb{CP}^1$ with vertices
$\gamma_0$, $\gamma_1$, $\dots$, $\gamma_n \in \mathbb{CP}^1$. We choose
an initial point $\tilde\gamma_0 \in \mathbb{CP}^1$ and prescribe a
cross-ratio $\mu \in \mathbb{C}$. The \emph{discrete Darboux transform
of $\gamma$ with initial point $\tilde\gamma_0$ and parameter $\mu$}
is the unique discrete curve $\tilde\gamma$ whose vertices $\tilde\gamma_j$,
$j = 1, 2, \dots, n$, satisfy
\[(\gamma_{j - 1}, \gamma_j; \tilde\gamma_j, \tilde\gamma_{j-1})
:= \frac{(\gamma_{j-1} - \tilde\gamma_j)(\gamma_j -
\tilde\gamma_{j-1})}{(\gamma_{j-1} - \tilde\gamma_{j-1})(\gamma_j -
\tilde\gamma_j)} = \mu.\]
\end{definition}

\begin{lemma}
\label{lem:moebius-step}
Let $a, b, d \in \mathbb{CP}^1$ be in general position. For every $\mu
\in \mathbb{C}$, there exists a Möbius transformation depending on $a,
b$, and $\mu$ that maps $d$ to $c \in \mathbb{CP}^1$ such that $(a, b;
c, d) = \mu$.
\end{lemma}

\begin{proof}
Consider the Möbius transformation \[M\colon x \mapsto \frac{x -
a}{x - b},\] which maps $a$, $b$, and $d$ to $0$, $\infty$, and $d'$
respectively.
The cross-ratio is invariant under Möbius transformations. Hence,
if we denote the image of $c$ under $M$ as $c'$, we want that
\[(0, \infty; c', d') = \frac{(0 - c')(\infty - d')}{(0 -
d')(\infty - c')} = \frac{c'}{d'} = \mu.\]
We define the Möbius transformations \[N\colon d' \mapsto c' = \mu d',
\quad M^{-1}\colon x' \mapsto \frac{b x' - a}{x' - 1}.\]
Then the Möbius transformation
\[M^{-1} \circ N \circ M\colon d \mapsto \frac{(\mu b - a) d - (\mu -
1) a b}{(\mu - 1) d + b - \mu a}\]
maps $d \in \mathbb{CP}^1$ to $c \in \mathbb{CP}^1$ such that $(a, b;
c, d) = \mu$.
\end{proof}

\noindent Note that $(M^{-1} \circ N \circ M)(a) = a$ and $(M^{-1}
\circ N \circ M)(b) = b$, independent of $\mu$.

\begin{lemma}
\label{lem:moebius-endpoint}
There exists a Möbius transformation depending on $\gamma_0$,
$\gamma_1$, $\dots$, $\gamma_n$, and $\mu$ that maps an initial point
$\tilde\gamma_0$ of a discrete Darboux transform of $\gamma$ with
parameter $\mu$ to the corresponding end point $\tilde\gamma_n$.
\end{lemma}

\begin{proof}
By \autoref{lem:moebius-step}, there exist Möbius transformations
$M_j$, $j = 1, 2, \dots, n$, depending on $\gamma_{j-1}$, $\gamma_j$,
and $\mu$ that map $\tilde\gamma_{j-1}$ to $\tilde\gamma_j$. Therefore
their composition $M_n \circ M_{n-1} \circ \dots \circ M_1$ is a Möbius
transformation depending on $\gamma_0$, $\gamma_1$, $\dots$, $\gamma_n$,
and $\mu$ that maps $\tilde\gamma_0$ to $\tilde\gamma_n$.
\end{proof}

\begin{remark}
A discrete Darboux transform $\tilde\gamma$ is closed if and only if
its initial point $\tilde\gamma_0$ is a fixed point of the Möbius
transformation of \autoref{lem:moebius-endpoint}.
The Möbius transformation of \autoref{lem:moebius-endpoint}
is of the form \[x \mapsto \frac{a x + b}{c x + d},\] where $a$, $b$,
$c$, and $d$ are polynomials in $\mu$ with complex coefficients depending
on $\gamma_0$, $\gamma_1$, $\dots$, $\gamma_n$.
Its fixed points are the roots of the equation \[(c x + d) x - (a x +
b) = c x^2 + (d - a) x - b = 0.\]
This equation is quadratic in $x$. Its degree in $\mu$ increases with
the number of points of $\gamma$.
Equivalently, in homogeneous coordinates, the fixed points are the
eigenvectors of matrix \[\begin{pmatrix}a & b\\c & d\end{pmatrix}.\]
\end{remark}

\begin{example}
\noindent As a simple but interesting enough example, consider the closed
discrete curve $\gamma$ spanned by the fifth roots of unity, \[\gamma_j =
\mathrm{e}^{2 \pi \mathrm{i} j/5}, \quad j = 0, 1, \dots, 5.\]
The relationship between $\mu$ and the initial point $\tilde\gamma_0$
of a closed discrete Darboux transform $\tilde\gamma$ of $\gamma$
is governed by the equation
\begin{equation}
\label{eq:pentagon}
\begin{split}
\left[\left(\left(-3 + \sqrt{5}\right) \mu^2 + 6 \mu - 3 - \sqrt{5}\right)
\tilde\gamma_0^2 + \left(\left(-2 - 4 \sqrt{5}\right) \mu + 1 +
\sqrt{5}\right) \tilde\gamma_0\right.\\\left.+ \left(-3 + \sqrt{5}\right)
\mu^2 + 6 \mu - 3 - \sqrt{5} \right] (1 - \mu) = 0.
\end{split}
\end{equation}

\noindent Equation~\eqref{eq:pentagon} is quadratic in $\tilde\gamma_0$,
cubic in $\mu$, and has total degree $5$.
For almost every value of $\mu$, exactly two
values of $\tilde\gamma_0$ satisfy the equation. The only exceptions
are $\mu = 1$, where all values of $\tilde\gamma_0$ satisfy the equation,
and $\mu = \frac{3 + \sqrt{5}}{8}$ and $\mu = \infty$, which are
ramification points of $\tilde\gamma_0$, i.e.\ points where there is
only one value of $\tilde\gamma_0$, which is a root of multiplicity $2$
of~\eqref{eq:pentagon}.

The discrete Darboux transform $\tilde\gamma$ of $\gamma$ with initial
point $\tilde\gamma_0 = \gamma_1$ and parameter $\mu = 0$ is identical
to $\gamma$ up to a rotation by $2 \pi / 5$, i.e.\ \[\tilde\gamma_{j-1} =
\mathrm{e}^{2 \pi \mathrm{i} / 5} \cdot \gamma_{j-1} = \gamma_j\] for
all $j = 0, 1, \dots, n$. Particularly, the discrete curve $\tilde\gamma$
is closed.

We would like to examine how the closed Darboux transform $\tilde\gamma$
behaves when $\mu$ makes two full anticlockwise turns around the
ramification point $\frac{3 + \sqrt{5}}{8}$ on a circle through the
origin centred at $\left(\frac{3 + \sqrt{5}}{8}\right)/2 + \frac{1}{1000}$.

\begin{figure}
\centering
\setcounter{subfigure}{0}
\begin{subfigure}[c]{0.49\textwidth}
\includegraphics{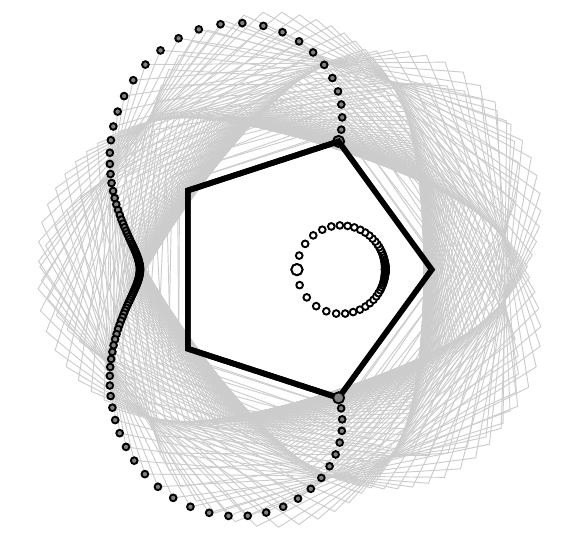}
\end{subfigure}
\begin{subfigure}[c]{0.49\textwidth}
\includegraphics{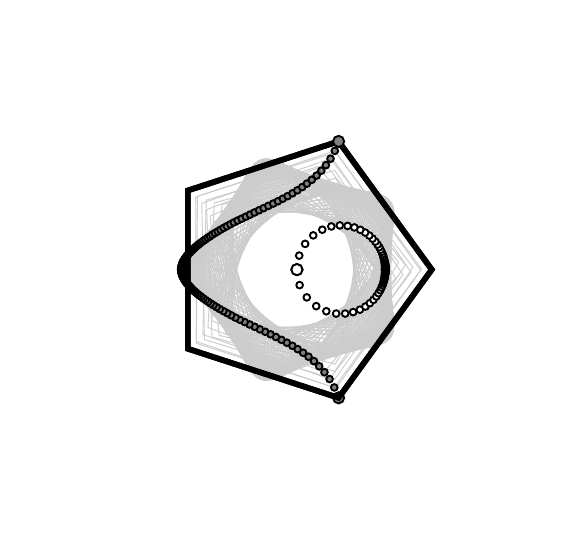}
\end{subfigure}
\caption{continuous deformation of a closed discrete Darboux transform}
\label{fig:pentagon-deformation}
\end{figure}

\autoref{fig:pentagon-deformation} attempts to illustrate the
behaviour of the closed Darboux transform $\tilde\gamma$ under the
aforementioned motion of $\mu$. (It is notoriously difficult to reproduce
dynamic behaviour on paper. A video of the experiment is available in
the supplementary material for this article.)
We can read \autoref{fig:pentagon-deformation} in two different
ways:

Firstly, the left image shows the movement of $\tilde\gamma_0$
(grey points) as $\mu$ (white points) completes one full circle. Then
the right image shows the movement of $\tilde\gamma_0$ (grey points)
as $\mu$ (white points) completes another full circle.
The position of the Darboux transform (black) after one turn of $\mu$
is identical to the initial position up to rotation. The final position
of the Darboux transform after the second turn is absolutely identical
to the initial position.

Secondly, the left image shows the movement of one choice of
$\tilde\gamma_0$ such that $\tilde\gamma$ is closed as $\mu$
completes one full circle. The right image shows how the other choice
of $\tilde\gamma_0$ such that $\tilde\gamma$ is closed moves at the
same time.
After one turn of $\mu$ we reach the initial position up to interchanged
choices of $\tilde\gamma_0$. (In the left image, $\tilde\gamma_0$ moves
from $\gamma_1$ to $\gamma_4$ while in the right image, $\tilde\gamma_0$
moves from $\gamma_4$ to $\gamma_1$.) After another turn of $\mu$,
the two choices of $\tilde\gamma_0$ reach their initial positions again.

Following \autoref{rem:analytic-continuation}, the steps
of $\mu$ are chosen according to the epsilon-delta bound
of \autoref{thm:epsilon-delta-bound} for~\eqref{eq:pentagon}
with $\varepsilon$ half the distance between the two choices of
$\tilde\gamma_0$. As we expect, the closer $\mu$ approaches the
singularity the smaller the steps become. At its rightmost position,
$\mu$ is only $\frac{1}{1000}$ away from the singularity. It takes
$127$ steps until $\mu$ completes one full circle.
\end{example}

\begin{remark}
If we want to prevent jumps, $\mu$ and $\tilde\gamma_0$ cannot
both be freely movable, i.e.\ we cannot let $\mu$ and $\tilde\gamma_0$
interchange their roles as movable and dependent point. Otherwise,
we can force a jump as follows:
We move parameter $\mu$ to $\mu = 1$. At the same time, according
to~\eqref{eq:pentagon}, the two possible initial points of $\tilde\gamma$
move to the origin and the point at infinity, respectively.
Without loss of generality, we assume that $\tilde\gamma_0$ moved to
the origin. Note that $\mu = 1$ describes an irreducible component
of the plane algebraic curve~\eqref{eq:pentagon}. This means that if
we remove $\tilde\gamma_0$ from the origin, $\mu$ will simply rest at
$\mu = 1$. Then we cannot move $\mu$ without a jump of $\tilde\gamma_0$
because in order to move continuously, $\tilde\gamma_0$ would initially
have to be arbitrarily close to the origin (or the point at infinity).
\end{remark}

\begin{remark}
Floating point arithmetic introduces rounding errors into the
computation of $\tilde\gamma_0$. This has a peculiar effect: If
$\tilde\gamma$ is closed, we know that $\tilde\gamma_0$ must be one
of the two fixed points of the Möbius transformation $M$ that maps
the initial point $\tilde\gamma_0$ of $\tilde\gamma$ to its last point
$\tilde\gamma_n$.
In general, a Möbius transformation has one attracting and one repelling
fixed point. When the fixed point is repelling, any numerical error in
its position is amplified by Möbius transformation $M$. Therefore, the
closed Darboux transform may (numerically) no longer be closed when
computed naively.
We have observed (see~\autoref{fig:pentagon-deformation}) that we can
move from one choice of $\tilde\gamma_0$ to the other by moving $\mu$
around a ramification point.
The natural domain for the map $c\colon \mu \mapsto \tilde\gamma_0$ is
a Riemann surface. Different choices of $\tilde\gamma_0$ correspond to
different branches of the Riemann surface. When we compute the vertices of
$\tilde\gamma$, we step by step compute $M \circ c$ = $M_n \circ M_{n-1}
\circ \dots \circ M_1 \circ c$. Note that function $M \circ c$ is an
example of a function on a Riemann surface that is numerically stable
on one branch and numerically unstable on the other.

Luckily, we can stabilize the computation by considering the inverse
Möbius transformation $M^{-1}$. A repelling fixed point of a Möbius
transformation is an attracting fixed point of its inverse. We can step by
step compute $M^{-1} \circ c = M_1^{-1} \circ M_2^{-1} \circ \dots \circ
M_n^{-1} \circ c$ to obtain the vertices of $\tilde\gamma$ in reversed
order. Since the cross-ratio of $A, B, C, D$ satisfies $(A,B;C,D) =
(B,A;D,C)$, we need only change our algorithm very little in order to
obtain $M^{-1} = M_1 \circ M_2 \circ \dots \circ M_n$ instead of $M =
M_n \circ M_{n-1} \circ \dots \circ M_1$; we only need to reverse the
order of the vertices of $\tilde\gamma$ before we compute the Möbius
transformations.

Besides, we can determine whether $\tilde\gamma_0$ approximates an
attracting fixed point of $M$ by considering the derivative of $M$ at
$\tilde\gamma_0$. The fixed point near $\tilde\gamma_0$ is attracting
if the absolute value of the derivative is smaller than $1$.
\end{remark}

\section[Towards triangular systems of polynomials]{Towards homotopy
continuation of triangular systems of polynomials}
\label{sec:towards-triangular-systems}

In this section, we discuss a scheme for certified homotopy continuation
of triangular systems of polynomials. A general implementation has
remained elusive so far. However, we later follow the same scheme when
we derive an algorithm for certified homotopy continuation of systems
of plane algebraic curves (\autoref{alg:plane-curves-system}).

\begin{problem}
\label{prob:triangular-system}
Consider a triangular system of polynomials, without loss of generality
\begin{equation}
\label{eq:triangular-system}
\begin{gathered}
p_1(x_0, x_1) = 0,\\
p_2(x_0, x_1, x_2) = 0,\\
\vdots\\
p_n(x_0, x_1, \dots, x_n) = 0.
\end{gathered}
\end{equation}
Let $x_0(0)$, $x_1(0)$, $\dots$, $x_n(0)$ be initial values that satisfy
the system of equations and let $x_0(1)$ be a target value for variable
$x_0$, i.e.\ a value to which $x_0$ should move continuously.
We define function $x_0(t)$ as a parameterization of the segment between
$x_0(0)$ and $x_0(1)$, \[x_0(t) = (1 - t) x_0(0) + t x_0(1).\]
By analytic continuation w.r.t.\ $t \in \left[0,1\right]$ we can (unless
there are singularities on the curves along which we perform analytic
continuation) step by step define holomorphic functions $x_1(t)$,
$x_2(t)$, $\dots$, $x_n(t)$. For example, we obtain $x_1(t)$ from
$p_1(x_0(t), x_1(t)) = 0$, then $x_2(t)$ from $p_2(x_0(t), x_1(t),
x_2(t)) = 0$, etc.

Compute the target values $x_1(1)$, $x_2(1)$, $\dots$, $x_n(1)$ from the
given polynomial system, all initial values and the first target value.
\end{problem}

\begin{remark}
\label{rem:fundamental-difficulty}
Any algorithm for this problem has to face the following fundamental
difficulty: Among all paths $x_j(t)$ along which we perform analytic
continuation to define the next path $x_k(t)$, generally only $x_0(t)$
is linear.
The other paths $x_1(t)$, $x_2(t)$, $\dots$, $x_n(t)$ are almost always
curvilinear---and unknown. We can at best evaluate $x_1(t)$, $x_2(t)$,
$\dots$, $x_n(t)$ at finitely many discrete points in time and interpolate
between the sample points. However, we must make sure that the approximate
paths we obtain by discretization remain close enough to the actual paths
such that they yield the same result w.r.t.\ analytic continuation. In
particular, we must make sure that in every step no singularities lie
between approximate and actual path. To make things worse, this includes
singularities of variables that occur only in later equations, whose
position in time may change depending on how we approximate the current
step.

One way to attack this difficulty is to eliminate $x_1$, $x_2$, $\dots$,
$x_{n - 1}$ from the polynomial system~\eqref{eq:plane-curves-system},
e.g.\ using resultants. However, this approach is expensive and suffers
from exponential expression swell. The resulting polynomial equation in
$x_0$, $x_n$ very likely has a high total degree, huge coefficients,
and many (artificial) critical points.
This means that we can in principle apply the
method for analytic continuation of plane algebraic curves
of~\autoref{alg:analytic-continuation} but that in practice it will
often be too expensive (see~\autoref{ex:artificial-singularity}). If
elimination introduces artificial critical points on the path of $x_0$,
\autoref{alg:analytic-continuation} does not even terminate.
\end{remark}

\noindent Instead we pursue the following idea:

\begin{remark}[General scheme for homotopy continuation of triangular systems]
We perform homotopy continuation of one equation after another,
interpolating linearly between sample points (using a time parameter in
the unit interval).
In order to obtain sample points on the actual paths of the variables,
we synchronize the time step. This means that we let all variables make
time steps of the same size.
We determine a step width such that analytic continuation by proximity
is possible (as in \autoref{rem:analytic-continuation}), and such
that the linearly interpolated paths between consecutive sample points
are equivalent to the actual paths of the variables w.r.t.\ analytic
continuation.
To fulfil the latter requirement, the step width must be small enough
such that there are no singularities between linearly interpolated paths
and actual paths.
We cannot foresee whether linear paths and actual paths enclose
singularities of variables that occur only in later equations. We must
determine whether this is the case when we later analytically continue the
respective variable. Should we find that we have `caught' a singularity,
we start over with a smaller step width.
Unless there are singularities on the actual paths of variables,
there is a small neighbourhood around the actual paths that is free of
singularities. Eventually, after finitely many reductions of step size,
the linear paths approximate the actual paths of the variables well
enough such that we do not encounter singularities anymore. Then we can
make one synchronized time step with all variables.
We proceed until we reach time $t = 1$.
\end{remark}

\section[Homotopy continuation of systems of plane algebraic
curves]{Certified homotopy continuation of systems of plane algebraic
curves}
\label{sec:plane-curves-system}

\noindent In full generality, it may be very difficult to decide
whether or not there are singularities between linearly interpolated
paths and actual paths.
(Among other things, we may want to ensure that the $(k-1)$-dimensional
discriminant locus of variable $x_k$ w.r.t.\ equation $p_k(x_0, x_1,
\dots, x_k) = 0$ does not intersect the polydisc around the last sample
point with radii lengths of the linear paths.)
Therefore, we restrict ourselves to systems of plane algebraic curves,
a special case of~\autoref{prob:triangular-system}.

\begin{problem}
\label{prob:plane-curves-system}
Consider a system of bivariate polynomials, without loss of generality
\begin{equation}
\label{eq:plane-curves-system}
\begin{gathered}
p_1(x_0, x_1) = 0,\\
p_2(x_1, x_2) = 0,\\
\vdots\\
p_n(x_{n-1}, x_n) = 0.
\end{gathered}
\end{equation}
Let $x_0(0)$, $x_1(0)$, $\dots$, $x_n(0)$ be initial values that satisfy
the system of equations and let $x_0(1)$ be a target value.
We define function $x_0(t)$ as a parameterization of the segment between
$x_0(0)$ and $x_0(1)$, \[x_0(t) = (1 - t) x_0(0) + t x_0(1).\]
By analytic continuation w.r.t.\ $t \in \left[0,1\right]$ we can (unless
there are singularities on the curves along which we perform analytic
continuation) step by step define holomorphic functions $x_1(t)$,
$x_2(t)$, $\dots$, $x_n(t)$. For example, we obtain $x_1(t)$ from
$p_1(x_0(t), x_1(t)) = 0$, then $x_2(t)$ from $p_2(x_1(t), x_2(t)) =
0$, etc.

Compute the target values $x_1(1)$, $x_2(1)$, $\dots$, $x_n(1)$ from the
given polynomial system, all initial values and the first target value.
\end{problem}

\noindent Before we describe an algorithm
for~\autoref{prob:plane-curves-system}, we need the following lemma.

\begin{lemma}
\label{lem:range-estimate}
Let $\mathcal{C}\colon f(x,y) = 0$, $x_1 \in \mathbb{C}$ be defined
as in \autoref{thm:epsilon-delta-bound}. Let $\varepsilon > 0$. Suppose
that we have determined $\delta > 0$ by the epsilon-delta bound of
\autoref{thm:epsilon-delta-bound} such that \[|y_j(x_1) - y_j(x_2)| <
\varepsilon\] for all holomorphic functions $y_j(x)$ that satisfy $f(x,
y_j(x)) = 0$ in a neighbourhood of $x_1$ and for all $x_2$ with $|x_1 -
x_2| < \delta$.

\noindent Then for all $x_2$ with $\delta' = |x_1 - x_2| < \delta$,
\[\varepsilon' = \frac{\delta'}{\delta} \cdot \varepsilon < \varepsilon\]
satisfies \[|y_j(x_1) - y_j(x_2)| < \varepsilon'.\]
\end{lemma}

\noindent This means that we can find a better estimate for the range of
$y_j(x)$ w.r.t.\ an actual feasible movement of $x$ from $x_1$ to $x_2$.

\begin{proof}
Under the assumptions of~\autoref{lem:range-estimate},
\[f(x) = \frac{y_j(\delta x + x_1) - y_j(x_1)}{\varepsilon}\]
is a holomorphic function from the open unit disk to the open unit disk.
By the maximum modulus principle, we know that there exists a point on
the boundary of the disk of radius $\delta'$ around $x_1$ where $|y_j(x)
- y_j(x_1)|$ is greater or equal than at any point $x$ with $|x - x_1|
< \delta'$.
Hence, there exists a point on the boundary of the disk of radius
$\frac{\delta'}{\delta}$ around the origin where $|f(x)|$ is greater or
equal than at any point $x$ with $|x| < \frac{\delta'}{\delta}$.
Schwarz lemma states that \[|f(x)| \leq |x|\] for all $x$ in the open
unit disk. Therefore
\[\frac{\varepsilon'}{\varepsilon} = \max_{t \in \left[0,1\right]}
\left|f\left(\frac{\delta'}{\delta} \cdot \mathrm{e}^{2 \pi \mathrm{i}
t}\right)\right| \leq \max_{t \in \left[0,1\right]}
\left|\frac{\delta'}{\delta} \cdot \mathrm{e}^{2 \pi \mathrm{i} t}\right|
= \frac{\delta'}{\delta},\]
and thus \[\varepsilon' \leq \frac{\delta'}{\delta} \cdot \varepsilon\]
for all $x_2$ with $|x_1 - x_2| < \delta' < \delta$.
\end{proof}

\begin{remark}
\label{rem:delta-epsilon-bound}
Alternatively, if we plug in $\delta = \delta'$ and $\varepsilon = \varepsilon'$
into~\eqref{eq:epsilon-delta-bound} and solve for $\varepsilon'$, we obtain
\[\varepsilon' = \delta' \left(\tilde{y} + \frac{M \delta'}{\rho (\rho -
\delta')}\right) < \varepsilon,\]
with $M$, $\rho$, $\tilde{y}$ as in the proof of
\autoref{thm:epsilon-delta-bound}.
This yields another better estimate for the range of $y_j(x)$
w.r.t.\ an actual feasible movement of $x$ from $x_1$ to $x_2$.
\end{remark}

\begin{algorithm}
\label{alg:plane-curves-system}
Consider the system of bivariate polynomials of
\autoref{prob:plane-curves-system} with initial values $x_0(0)$, $x_1(0)$,
$\dots$, $x_n(0)$ and a target value $x_0(1)$.

\begin{enumerate}
\item Define $x_0(t) = (1 - t) x_0(0) + t x_0(1)$.
\item Let $T = 1$.
\item Let $\varepsilon_0' = |x_0(0) - x_0(T)|$.
\item For all $k = 1, 2, \dots, n$:
\begin{enumerate}
\item Let $\varepsilon_k$ be half the minimum distance between the $x_k$
with \[p_k(x_{k-1}(0), x_k) = 0.\]
\item Compute $\delta_k$ according to the epsilon-delta bound of
\autoref{thm:epsilon-delta-bound} for $f(x, y) = p_k(x_{k-1}, x_k)$,
$x_1 = x_{k-1}(0)$ and $\varepsilon = \varepsilon_k$.
\item If $\delta_k < \varepsilon_{k-1}'$ then let $T = T / 2$ and go to 3.
\item Let $x_k(T)$ be the $x_k$ with $p_k(x_{k-1}(T), x_k) = 0$ closest
to $x_k(0)$.
\item Let $\delta_k' = |x_{k-1}(0) - x_{k-1}(T)|$.
\item Let $\varepsilon_k' = (\delta_k' + \epsilon) / \delta_k \cdot
\varepsilon_k$
with $\epsilon > 0$.
\end{enumerate}
\item If $T = 1$ then output $x_1(T), x_2(T), \dots, x_n(T)$ and stop.
\item Let $x_0(0) = x_0(T)$, $x_1(0) = x_1(T)$, $\dots$, $x_n(0) =
x_n(T)$ and go to 1.
\end{enumerate}
\end{algorithm}

\begin{theorem}
If the target values $x_1(1)$, $x_2(1)$, $\dots$, $x_n(1)$
of \autoref{prob:plane-curves-system} are well-defined,
\autoref{alg:plane-curves-system} computes them in finitely many steps.
\end{theorem}

\begin{proof}
The first two steps of~\autoref{alg:plane-curves-system} are
initialization steps.
In step~1, we define a linear homotopy between initial value $x_0(0)$
and final value $x_0(1)$ of $x_0$.
We first want to test whether we can perform analytic continuation of
the system in a single time step. Therefore, in step~2, we set target
time $T = 1$.

Steps~3--6 form the main loop of our algorithm. They are repeated until
we reach time $T = 1$, in which case step~5 terminates the algorithm.

In step~3, we estimate the range of $x_0$ as it runs from its initial
position $x_0(0)$ to its target position $x_0(T)$. Since $x_0(t)$ is
linear by definition (step~1), our estimate $\varepsilon_0' = |x_0(0)
- x_0(T)|$ is exact.

Step~4 is the inner loop of our algorithm, in which we try to perform
analytic continuation equation by equation of our system. Run variable $k$
denotes the index of the equation $p_k(x_{k-1}, x_k)$ under consideration.

In steps~4a--4b, we use the epsilon-delta
bound of~\autoref{thm:epsilon-delta-bound} and
\autoref{rem:analytic-continuation} to compute a feasible step width
$\delta_k$ for variable $x_{k-1}$. If $x_{k-1}$ moves at most $\delta_k$
then we can perform analytic continuation of $x_k$ w.r.t.\ $p_k(x_{k-1},
x_k) = 0$ by selecting as $x_k(T)$ the value of $x_k$ with
$p_k(x_{k-1}(T), x_k) = 0$ closest to $x_k(0)$.

Hence, in step~4c, we test whether feasible step $\delta_k$ is smaller than
an upper bound $\varepsilon_{k-1}'$ of the range of $x_{k-1}$ as it runs
from $x_{k-1}(0)$ to $x_{k-1}(T)$.

If $\delta_k < \varepsilon_{k-1}'$, we cannot be sure that there are no
singularities between the actual path of $x_{k-1}$ and the interpolated
path, i.e.\ the segment from $x_{k-1}(0)$ to $x_{k-1}(T)$. Our attempt
to reach target time $T$ in one step has failed. Therefore, we halve
target time $T$ and go back to step~3.

Otherwise, if $\delta_k \geq \varepsilon_{k-1}'$, the epsilon-delta
bound of \autoref{thm:epsilon-delta-bound} guarantees that actual path
and interpolated path of $x_{k-1}$ are equivalent w.r.t.\ analytic
continuation of $x_k$. Then in step~4d, we determine target value
$x_k(T)$. By construction, $x_k(T)$ is a point on the actual path
of $x_k$.

In steps~4e--4f, we use~\autoref{lem:range-estimate} to compute an upper
bound for the range of $x_k$ as it runs from $x_k(0)$ to $x_k(T)$. The
computation is independent of whether $x_{k-1}$ runs along actual or
interpolated path. The bound $\varepsilon_k'$ holds for both paths,
particularly also for analytic continuation of $x_k$ along the actual
path of $x_{k-1}$.

We then proceed with analytic continuation of the next variable, if any.
When we leave the inner loop (step~4), we obtain valid positions for $x_1,
x_2, \dots, x_n$ at target time $T$.
If $T = 1$, we output the solution and stop (step~5). Otherwise, we use
$x_0(T), x_1(T), \dots, x_n(T)$ as a valid initial position from which
we again try to reach target time $T = 1$ (step~6).

By the assumption that $x_1(1), x_2(1), \dots, x_n(1)$ are well-defined,
there are only finitely many singularities in a neighbourhood of the
actual paths of $x_0$, $x_1$, $\dots$, $x_n$. The algorithm terminates
after finitely many steps as eventually the interpolated paths of $x_1$,
$x_2$, $\dots$, $x_n$ approximate the actual paths well enough such that
we do not encounter singularities anymore.
\end{proof}

\section{Comparison with other approaches}

Let us discuss more examples, which allow us to compare the performance
of our algorithm with that of other approaches.
(The number of steps needed by~\autoref{alg:analytic-continuation}
and~\autoref{alg:plane-curves-system} stated below relate to an
experimental implementation in Haskell that is available in the
supplementary material for this article.)

\begin{example}[\ocite{HauensteinHaywoodLiddell2014}*{Section~7.1}]
Consider the Newton homotopy
\[H(x, t) = f (x) + v t\]
where $f (x) = x^2 - 1 - m$ and $v = m$ for various values of $m
> -1$. The goal is to analytically continue $x$ as $t$ moves from $1$
to $0$.

In~\autoref{tab:selected-m} and~\autoref{tab:selected-k}, we
compare the performance of \autoref{alg:analytic-continuation}
with that of the algorithms of~\ocite{BeltranLeykin2013}
and~\ocite{HauensteinHaywoodLiddell2014}, for various
values of $m$. The data for the latter algorithms is quoted
from~\cite{HauensteinHaywoodLiddell2014}*{Table~1 and Table~2}.

\begin{table}[ht]
\centering
\begin{tabular}{@{}cccc@{}} \toprule
$m$ & \begin{minipage}{0.25\linewidth}Number of steps\\
of~\autoref{alg:analytic-continuation}\end{minipage}
& \begin{minipage}{0.25\linewidth}Number of
a priori steps of~\ocite{BeltranLeykin2013}\end{minipage} &
\begin{minipage}{0.25\linewidth}Number of a posteriori certified
intervals of~\ocite{HauensteinHaywoodLiddell2014}\end{minipage}\\ \midrule
10    &  9 & 184 &  51\\
20    & 12 & 217 &  67\\
30    & 14 & 237 &  78\\
40    & 16 & 250 &  82\\
50    & 17 & 260 &  88\\
60    & 18 & 269 &  92\\
70    & 19 & 276 &  96\\
80    & 20 & 282 &  99\\
90    & 21 & 288 & 103\\
100   & 21 & 292 & 105\\
1000  & 41 & 395 & 162\\
2000  & 49 & 426 & 180\\
3000  & 54 & 446 & 191\\
4000  & 58 & 457 & 197\\
5000  & 62 & 468 & 204\\
10000 & 73 & 499 & 220\\
20000 & 87 & 530 & 238\\
30000 & 96 & 547 & 250\\ \bottomrule
\end{tabular}
\caption{Performance of~\autoref{alg:analytic-continuation} in comparison
with the algorithms of~\ocite{BeltranLeykin2013}
and~\ocite{HauensteinHaywoodLiddell2014}, for various values
of $m$. The data in the last two columns is quoted
from~\cite{HauensteinHaywoodLiddell2014}*{Table~1}.}
\label{tab:selected-m}
\end{table}
\begin{table}[!ht]
\centering
\begin{tabular}{@{}cccc@{}} \toprule
$k$ & \begin{minipage}{0.25\linewidth}Number of steps\\
of~\autoref{alg:analytic-continuation}\end{minipage}
& \begin{minipage}{0.25\linewidth}Number of a
priori steps of~\ocite{BeltranLeykin2013}\end{minipage} &
\begin{minipage}{0.25\linewidth}Number of a posteriori certified
intervals of~\ocite{HauensteinHaywoodLiddell2014}\end{minipage}\\ \midrule
 1 &  5 &  176 & 64\\
 2 &  9 &  287 & 68\\
 3 & 14 &  390 & 70\\
 4 & 18 &  492 & 71\\
 5 & 22 &  593 & 71\\
 6 & 27 &  695 & 71\\
 7 & 31 &  798 & 71\\
 8 & 36 &  901 & 71\\
 9 & 40 & 1003 & 71\\
10 & 44 & 1108 & 71\\ \bottomrule
\end{tabular}
\caption{Performance of~\autoref{alg:analytic-continuation} in comparison
with the algorithms of~\ocite{BeltranLeykin2013}
and~\ocite{HauensteinHaywoodLiddell2014}, for various values of $m =
-1 + 10^{-k}$. The data in the last two columns is quoted
from~\cite{HauensteinHaywoodLiddell2014}*{Table~2}.}
\label{tab:selected-k}
\end{table}

\noindent Both~\ocite{BeltranLeykin2013}
and~\ocite{HauensteinHaywoodLiddell2014} present algorithms designed
for certified homotopy continuation of arbitrary polynomial systems
whereas \autoref{alg:analytic-continuation} can only deal with plane
algebraic curves.
However, the example indicates that in the univariate case
\autoref{alg:analytic-continuation} may perform much better than those
more general algorithms, which do not exploit the special structure of
the univariate case.
\end{example}

\noindent Furthermore, let us elaborate
on~\autoref{rem:fundamental-difficulty}. The following example shows
that it may be better to apply~\autoref{alg:plane-curves-system} to
a system of plane algebraic curves than to eliminate variables and
apply~\autoref{alg:analytic-continuation} to the resultant.

\begin{example}
\label{ex:artificial-singularity}
Consider the system of bivariate polynomials
\begin{equation}
\label{eq:artificial-singularity}
\begin{aligned}
p_1(x_0, x_1) &= -4 + 2 x_0 + x_1 + 2 x_0 x_1 + x_1^2 = 0,\\
p_2(x_1, x_2) &= x_1^2 + x_2^3 = 0,
\end{aligned}
\end{equation}
with initial values
\[x_0(0) = 0, \quad x_1(0) = \frac{-1 - \sqrt{17}}{2}, \quad x_2(0) =
{\left(\frac{-9 - \sqrt{17}}{2}\right)}^{\frac{1}{3}},\]
and target value $x_0(1) = 1$.
The $x_1$-resultant of $p_1(x_0, x_1)$ and $p_2(x_1, x_2)$ is
\begin{equation}
\label{eq:artificial-singularity-resultant}
q(x_0, x_2) = 16 - 16 x_0 + 4 x_0^2 + 9 x_2^3 + 4 x_0^2 x_2^3 + x_2^6 = 0.
\end{equation}

\noindent Let us compare the performance of
\autoref{alg:plane-curves-system} for~\eqref{eq:artificial-singularity}
with the performance of \autoref{alg:analytic-continuation}
for~\eqref{eq:artificial-singularity-resultant} as $x_0$ moves linearly
(in the unit interval) from $0$ to $1$.
We find that \autoref{alg:plane-curves-system} subdivides once, i.e.\
it needs two steps. In contrast, \autoref{alg:analytic-continuation}
needs six steps.
One possible explanation is that $x_0 = -\frac{1}{2}$ is a singularity
of~\eqref{eq:artificial-singularity-resultant} but not
of~\eqref{eq:artificial-singularity}. For $x_0 =
-\frac{1}{2}$,~\eqref{eq:artificial-singularity-resultant} has three
zeros of multiplicity two, whereas~\eqref{eq:artificial-singularity}
has six simple roots. Each zero of multiplicity two
of~\eqref{eq:artificial-singularity-resultant} corresponds to two simple
zeros of~\eqref{eq:artificial-singularity}
with differing signs of $x_1$.

Generally, elimination introduces artificial singularities. Due to an
artificial singularity it can even happen that we cannot analytically
continue the resultant: For example, the $x_1$-resultant of
\begin{align*}
\tilde{p}_1(x_0, x_1) &= -4 + 2 x_0 + x_2 - 2 x_0 x_1 + x_1^2 = 0,\\
p_2(x_1, x_2) &= x_1^2 + x_2^3 = 0,
\end{align*}
has an artificial singularity at $x_0 = \frac{1}{2}$. In this
case, \autoref{alg:analytic-continuation} does not terminate whereas
\autoref{alg:plane-curves-system} produces the desired result.
\end{example}

\section{Conclusion}

From an epsilon-delta bound for plane algebraic curves
(\autoref{thm:epsilon-delta-bound}), we have derived algorithms
for certified homotopy continuation of plane algebraic curves
(\autoref{alg:analytic-continuation}) and systems of plane algebraic
curves (\autoref{alg:plane-curves-system}).
Our certificate is rigorous for exact real arithmetic. For floating point
arithmetic, \autoref{thm:epsilon-delta-bound} can be considered a soft
certificate.
Several examples demonstrate the practicability of our approach.

A generalization of \autoref{alg:plane-curves-system} to arbitrary systems
of polynomials might be an interesting challenge for further research.

\section*{Acknowledgements}

I thank Ulrich Bauer, Tim Hoffmann, Jürgen Richter-Gebert, Katharina
Schaar, and Martin von Gagern for their support in preparing this article.

\section*{Funding}

This research was supported by DFG Collaborative Research Center TRR 109,
``Discretization in Geometry and Dynamics''.

\end{document}